\documentclass[12pt]{amsart}

 \usepackage{amsfonts,graphics,amsmath,amsthm,amsfonts,amscd,amssymb,amsmath,latexsym,multicol}
\usepackage{epsfig,url}
\usepackage{flafter}

\makeatletter

\def\jobis#1{FF\fi
  \def\predicate{#1}%
  \edef\predicate{\expandafter\strip@prefix\meaning\predicate}%
  \edef\job{\jobname}%
  \ifx\job\predicate
}

\makeatother

\if\jobis{proposal}%
\else
\fi

 \usepackage[matrix, arrow]{xy}

\DeclareMathOperator{\Supp}{Supp}

\DeclareMathOperator{\Pic}{Pic}

 \newcommand{\C}{\mathbb C}

 \newcommand{\Q}{\mathbb Q}
 \newcommand{\R}{\mathbb R}
 
 \newcommand{\bir}{\dashrightarrow}
 \newcommand{\rddown}[1]{\left\lfloor{#1}\right\rfloor} % round-down

 \numberwithin{equation}{subsection}
 \numberwithin{footnote}{subsection}

 \newtheorem{thm}[subsection]{Theorem}
 \newtheorem{conj}[subsection]{Conjecture}

{

{%_%_%_%_% upright style; roman (non-italic) text
    \newtheoremstyle{upright}%
        {8pt plus2pt minus4pt}%
        {8pt plus2pt minus4pt}%
        {\upshape}%
        {}%
        {\bfseries\scshape}%
        {}%
        {1em}%
        {}%
\theoremstyle{upright}

 \newtheorem{rem}[subsection]{Remark}

}

\title[\large V\MakeLowercase{arieties fibred over abelian varieties}]
{\large V\MakeLowercase{arieties fibred over abelian varieties with fibres of log general type}}
\author{\large C\MakeLowercase{aucher} B\MakeLowercase{irkar and} 
J\MakeLowercase{ungkai} A\MakeLowercase{lfred} C\MakeLowercase{hen}}

\date{\today}
\begin{document}
\maketitle

\begin{abstract}
Let $(X,B)$ be a complex projective klt pair, and let $f\colon X\to Z$ be a surjective 
morphism onto a normal projective variety with maximal albanese dimension such that 
$K_X+B$ is relatively big over $Z$. We show that such pairs have good log minimal models.
\end{abstract}

%\tableofcontents
%%%%%%%%%%%%%%%%%%%%%%%%

\section{Introduction}

We work over $\C$. The following is the main theorem of this paper.

\begin{thm}\label{t-mmodel}
Suppose that $(X,B)$ is a projective klt pair with $B$ a $\Q$-boundary and that we are given a surjective 
morphism $f\colon X\to Z$ 
where $Z$ is a normal projective variety with maximal albanese dimension (eg an abelian variety).
If $K_X+B$ is big$/Z$, then $(X,B)$ has a good log minimal model. Moreover, 
if $F$ is a general fibre of $f$, then 
$$
\kappa(K_X+B)\ge \kappa(K_F+B_F)+\kappa(Z)=\dim F+\kappa(Z)
$$
where $K_F+B_F=(K_X+B)|_F$ and $\kappa(Z)$ means the Kodaira dimension of a smooth model 
of $Z$.
\end{thm}

The proof is given in Section 5. Here we briefly outline the main steps of the 
proof. Since $Z$ has maximal albanese dimension, it admits a generically finite morphism 
$Z\to A$ to an abelian variety. We can run an appropriate  LMMP on $K_X+B$ (and on related log divisors) 
over $A$ by [\ref{BCHM}]. Since $K_X+B$ is big$/Z$, such LMMP terminate. The log minimal 
model we obtain is actually globally a log minimal model (that is, not only over $A$) because $A$ is an abelian variety 
(see Section 3). It requires a lot more work to show that the log minimal model is \emph{good}, 
that is, to show that abundance holds for it. Such semi-ampleness statements are 
usually proved by establishing a nonvanishing result, creating log canonical centres, and 
lifting sections from a centre. For the nonvanishing part, 
we use some of the ideas of Campana-Chen-Peternell [\ref{CCP}, Theorem 3.1] to show that 
$K_X+B\equiv D$ for some $\Q$-divisor $D\ge 0$. This uses the Fourier-Mukai transform 
of the derived category of $A$ combined with the above LMMP and vanishing theorems  
(see Section 4). One can replace $K_X+B\equiv D$ with $K_X+B\sim_\Q D$ [\ref{CKP}][\ref{Kawamata-v=0}]. 
Next, we create a log canonical centre and lift sections from it using 
the extension theorem of Demailly-Hacon-P\v{a}un [\ref{DHP}, Theorem 1.8]. So, 
our proof is somewhat similar to the proof of the base point free theorem but 
the tools we use are quite different. 

Theorem \ref{t-mmodel} is interesting for several reasons. First, it is a non-trivial special case of 
the minimal model and abundance conjectures. Second, its proof provides an   
application of the extension theorem of [\ref{DHP}] which is hard to apply in other contexts because of 
relatively strong restrictions on log canonical centres. Third,  
it is closely related to the following 
generalized Iitaka conjecture. 

\begin{conj}\label{conj-iitaka-pairs}
Let  $(X,B)$ be a projective klt pair with $B$ a $\Q$-boundary, and let
$f\colon X\to Z$ be an algebraic fibre space onto a smooth projective variety $Z$. Then
$$
\kappa(K_X+B)\ge \kappa(K_F+B_F)+\kappa(Z)
$$
where $F$ is a general fibre of $f$ and $K_F+B_F:=(K_X+B)|_F$
\end{conj}

It is plausible that our proof of Theorem \ref{t-mmodel} can be extended to give a proof 
of the conjecture when $Z$ has maximal albanese dimension. In order to do so we would 
need an extension theorem stronger than [\ref{DHP}, Theorem 1.8] (see Section 6 for more details).
When $X$ is smooth, $B=0$, and $Z$ has maximal albanese dimension, the conjecture was 
proved by Chen and Hacon [\ref{Chen-Hacon}] using very different methods.

\subsection*{{Acknowledgements.}} 
The authors would like to thank Hajime Tsuji and Mihai P\v{a}un for many discussions on 
Conjecture \ref{conj-iitaka-pairs}.
The first author was supported by the 
Fondation Sciences Math\'ematiques de Paris when he visited Jussieu Mathematics Institute in Paris, 
and by a Leverhulme grant. 
Part of this work was done while 
the first author was visiting the second author at the National Taiwan University, and he would like 
to thank everyone for their hospitality.

%%%%%%%%%%%%%%%%%%%%%%
%%%%%%%%%%%%%%%%%%%%%%
\vspace{0.5cm}
\section{Preliminaries}

\subsection{{Notation and conventions.}} We work over the complex numbers $\C$
and our varieties are quasi-projective unless stated otherwise. A pair $(X,B)$ 
consists of a normal variety $X$ and a $\Q$-divisor $B$ with coefficients in 
$[0,1]$ such that $K_X+B$ is $\Q$-Cartier. For definitions and basic properties 
of singularities of pairs such as log canonical 
(lc), Kawamata log terminal (klt), dlt, plt, and the log minimal model program 
(LMMP) we refer to [\ref{Kollar-Mori}].

\subsection{{Divisors and Iitaka fibrations.}} Let $f\colon X\to Z$ be a 
projective morphism of normal varieties, and $Z'$ the normalization 
of the image of $f$. 
By a general fibre of $f$ we mean a general fibre of the induced morphism
$X\to Z'$. 
We say that an  $\R$-divisor $D$ on $X$ is big$/Z$ if $D\sim_\R H+E/Z$ where $H$ is an
ample$/Z$ $\R$-divisor and $E$ is an effective $\R$-divisor. Note that $D$ is big$/Z$
iff it is big$/Z'$.

Now assume  that $D$ is $\Q$-Cartier. When $X$ is projective, Iitaka [\ref{Iitaka-book}, \S 10]
proves the existence of the so-called \emph{Iitaka fibration} of $D$ assuming $\kappa(D)\ge 0$.
More precisely, there is a resolution $h \colon W\to X$ and a contraction $g\colon W\to T$
such that $\kappa(D)=\dim T$ and $\kappa(h^*D|_G)=0$ for a general fibre $G$ of $g$.
A similar construction exists for the map $f$ even if $X$ is not projective. That is,
 there is a resolution $h \colon W\to X$ and a contraction $g\colon W\to T/Z$
such that $\kappa(D|_F)=\dim V$ and $\kappa(h^*D|_G)=0$
where $F$ is a general fibre of $f$, $G$ is a general fibre of $g$, and
$V$ is a general fibre of $T\to Z$. This is called the \emph{relative Iitaka fibration} and
its existence follows from the absolute projective case: first take an appropriate compactification
of $X,Z$ and then consider the Iitaka fibration of the divisor $D+f^*A$
where $A$ is a sufficiently ample divisor on $Z$, then apply [\ref{Iitaka-book}, \S 10, Exercise 10.1].

\subsection{{Minimal models.}}
Let $f\colon X\to Z$ be a 
projective morphism of normal varieties and $D$ an $\R$-Cartier divisor 
on $X$. A normal variety $Y/Z$ together with a birational map $\phi\colon X\bir Y/Z$ 
whose inverse does not contract any divisors is called a \emph{minimal model of $D$ over $Z$} if: 

$\bullet$ $D_Y=\phi_*D$ is nef$/Z$,

$\bullet$ there is a common resolution $g\colon W\to X$ and $h\colon W\to Y$ such that 
$E:=g^*D-h^*D_Y$ is effective and $\Supp g_*E$ contains all the exceptional divisors of $\phi$.

Moreover, we say that $Y$ is a \emph{good} minimal model if $D_Y$ is semi-ample$/Z$.

If one can run an LMMP on $D$ which terminates with a model $Y$ on which $D_Y$ is nef$/Z$, 
then $Y$ is a minimal model of $D$ over $Z$. When $D$ is a log divisor, that is of the form 
$K_X+B$ for some pair $(X,B)$, we use the term \emph{log minimal model}.

%%%%%%%%%%%%%%%%%%%%%%
\vspace{0.5cm}

\section{{LMMP over abelian varieties.}}\label{rem-LMMP-abelian} 

In this section, we discuss running the LMMP over an abelian variety which we need in the later 
sections. 
 
\subsection{{Relative LMMP for polarized pairs of log general type.}}\label{ss-relative-LMMP}
Let $f\colon X\to Z$ be a morphism from a
normal projective variety $X$.
Let $D$ be a $\Q$-Cartier divisor on $X$ such that $D\sim_\Q K_X+B+L$ where we assume $D$ is big$/Z$,
$L$ is nef globally, and $(X,B)$ is klt. Sometimes $(X,B+L)$ is referred to as a \emph{polarized pair}. 
Since $D$ is big$/Z$, we can write $D\sim_\Q H+E/Z$ where $H$ is an ample $\Q$-divisor (globally)
and $E\ge 0$. For each sufficiently small rational number $\delta>0$ there is a log divisor  
$K_X+B_{\delta}$ such that $(X,B_{\delta})$ is klt and 
over $Z$ we have
$$
K_X+B_{\delta} \sim_\Q  K_X+B+\delta E+L+\delta H \sim_\Q 
(1+\delta)(K_X+B+L)
 /Z
$$
where we make use of the facts that $L+\delta H$ is ample and $(X,B+\delta E)$ is klt.
Now by [\ref{BCHM}],  we can run an LMMP$/Z$ on $K_X+B_{\delta}$ ending up with 
a log minimal model $Y/Z$. The LMMP is also an LMMP on $K_X+B+L$ so $Y$ is also a minimal model$/Z$ of 
$K_X+B+L$. Moreover, $K_Y+B_Y+L_Y$ is semi-ample/$Z$ by applying the base point free 
theorem to $K_Y+B_{\delta,Y}$.

\subsection{{LMMP over an abelian variety}}\label{ss-LMMP-abelian}
In addition to the assumptions of \ref{ss-relative-LMMP}, suppose that $Z=A$ is an abelian variety.
Then,
$K_Y+B_Y+L_Y$ is nef, not only$/A$, but
it is so globally. If this is not the case, then there would be an extremal ray $R$ on $Y$ such that
 $(K_{Y}+B_{Y}+L_{Y})\cdot R<0$. Then  $(K_{Y}+\Theta_{Y})\cdot R<0$ for some
 $\Theta$
 where
 $$
 K_X+\Theta\sim_\Q K_X+B+L+\tau E
 $$
  for some sufficiently
 small $\tau>0$ and such that $K_X+\Theta$ and $K_{Y}+\Theta_{Y}$ are both klt.
Thus there is a rational
curve generating $R$ which is not contracted$/A$. This is not possible as $A$ is an abelian variety.

\subsection{{Relative LMMP for certain dlt pairs}}\label{ss-LMMP-dlt}
In addition to the assumptions of \ref{ss-relative-LMMP}, suppose  that $D\ge 0$,  
$\Delta:=B+N$ for some $\Q$-Cartier $\Q$-divisor $N\ge 0$ with  $\Supp N=\Supp D$, and
that $(X,\Delta)$ is lc. We show that we can run an LMMP$/Z$ on $K_X+\Delta+L$ ending up 
with a minimal model.
By assumptions, 
$$
\Supp \rddown{\Delta}\subseteq\Supp N= \Supp D
$$
So, for each sufficiently small rational number $\epsilon>0$, we can write 
\begin{equation*}
\begin{split}
K_X+B'+L &:= K_X+B+N-\epsilon D-\epsilon N+L\\
&\sim_\Q (1-\epsilon)(K_X+\Delta+L)  
\end{split}
\end{equation*}
where $K_X+B'$ is klt. So, by (1) we can run an LMMP$/Z$ on $K_X+B'+L$ which ends up 
with a minimal model$/Z$ of $K_X+B'+L$ hence a minimal model$/Z$ of $K_X+\Delta+L$.

If $Z=A$ is an abelian variety, then by \ref{ss-LMMP-abelian} the minimal model just 
constructed is global not just over $A$.

%%%%%%%%%%%%%%%%%%%%%%
\vspace{0.5cm}
\section{The nonvanishing}

The main result of this section is the following nonvanishing statement.

\begin{thm}\label{t-nv}
Let $f\colon X\to A$ be a morphism from a normal projective variety $X$ to an abelian variety $A$.
Assume that

$(1)$ $D=K_X+B+L$ is big$/A$, $L$ is globally nef, and

$(2)$ $(X,B)$ is klt.

 Then, there is $P\in \Pic^0_\Q(A)$ such that $\kappa(D+f^*P)\ge 0$.
\end{thm}

\begin{proof} \emph{Step 1.} By Section 3, there is a sequence
$X=X_1\bir X_2 \bir \cdots X_r=X'/A$ of divisorial contractions and log flips
with respect to $K_X+B+L$ so that $K_{X'}+B_{X'}+L_{X'}$ is nef.
Let $I$ be a positive integer so that $ID_{X'}$ is Cartier, and
put $\mathcal{F}_s:=f'_*\mathcal{O}_{X'}(sID_{X'})$ where $f'$ is the
map $X'\to A$.
For any ample divisor $H$ on $A$ and any $P \in \rm{Pic}^0(A)$,
by the Kawamata-Viehweg vanishing theorem,
$$
R^if_*'\mathcal{O}_{X'}(sID_{X'}+ f^{'*}H + f^{'*}P)=0 ~~ \mbox{for all $i>0$ and $s\ge 2$}
$$
and
$$
H^i(X',sID_{X'}+ f'^*H + f'^*P)=0 ~~ \mbox{for all $i>0$ and $s\ge 2$}
$$
hence
$$
H^i(A,\mathcal{F}_s\otimes \mathcal{O}_{A}(H+P))=0
$$
for all $i >0$ and $s\ge 2$.
In other words, the sheaf $\mathcal{F}_s\otimes \mathcal{O}_{A}(H)$ is $IT^0$ for all ample divisors $H$.

Let $\phi\colon \tilde{A}\to A$ be a projective \'etale map, and define $\tilde{X}_i=\tilde{A}\times_A X_i$.
Then, $K_{\tilde{X}}+\tilde{B}$, the pullback of $K_X+B$,
is also klt, and $\tilde{L}$, the pullback of $L$, is nef.
The birational map $\tilde{X}\bir\tilde{X}'/\tilde{A}$ is decomposed into a sequence 
of divisorial contractions and log flips with respect to $K_{\tilde{X}}+\tilde{B}+\tilde{L}$ and
$K_{\tilde{X}'}+\tilde{B}_{\tilde{X}'}+\tilde{L}_{\tilde{X}'}$ is nef. Therefore, we have vanishing
properties on $\tilde{X}$ and $\tilde{A}$ similar to those above.

From now on, essentially we only need the above vanishing properties so to simplifiy
notation we replace $X$ by $X'$, $K_X+B$ by $K_{X'}+B_{X'}$, and $D$ by $D_{X'}$:
the new $D$ is nef but of course we may loose
the nefness of $L$.\\

\emph{Step 2.}
It turns out that $ \mathcal{F}_s$
satisfies a {\it generic vanishing theorem} for $s\ge 2$. That is
to say that we claim that $ \mathcal{F}_s$ satisfies the following condition
of Hacon [\ref{Ha04}, Theorem 1.2]:

$(*)$ \emph{let $\hat{M}$ be any sufficiently ample line bundle
on the dual abelian variety $\hat{A}$, and let $M$ on $A$ be the Fourier-Mukai transform of
$\hat{M}$,  and ${M}^\vee$ the dual of $M$; then $H^i(A,\mathcal{F}_s\otimes {M}^\vee)=0$
for all $i>0$.}\\

The condition $(*)$ implies the existence of a chain of inclusions
$$
V^0( \mathcal{F}_s) \supset V^1(\mathcal{F}_s)\supset\cdots \supset
V^n(\mathcal{F}_s)
$$
where
$$
V^i(\mathcal{F}_s):=\{ P \in {\rm
Pic}^0(A) \mid H^i(A, \mathcal{F}_s \otimes P) \ne 0\}
$$
and $n=\dim A$ (see the comments following [\ref{Ha04}, Theorem 1.2]).\\

\emph{Step 3.} Assume that $(*)$ holds.  Since
$\mathcal{F}_s$ is a sheaf of rank $h^0(F, sID|_F)$ where $F$ is a general fibre of $f$, it
is a non-zero sheaf for $s \gg 0$ because $D$ is big over $A$. Therefore, 
 $V^0(\mathcal{F}_s) \ne\emptyset$ for $s \gg 0$ otherwise $V^i(\mathcal{F}_s) =\emptyset$ for all
$i$ which implies that the Fourier-Mukai transform of
$\mathcal{F}_s$ is zero (by base change and the definition of
Fourier-Mukai transform). This contradicts  the fact that
the Fourier-Mukai transform
gives an equivalence of the derived categories $D(A)$ and $D(\hat{A})$ as proved by 
Mukai [\ref{Mukai}, Theorem 2.2]. 
Now $V^0(\mathcal{F}_s) \ne \emptyset$ implies that
$$
H^0(X,sID+f^*P)=H^0(A, \mathcal{F}_s\otimes \mathcal{O}_A(P)) \ne 0
$$
for some $P \in \Pic^0(A)$. This implies the theorem.\\

\emph{Step 4.} It remains to prove the claim $(*)$ above.
Let  $\phi: \hat{A} \to A$ be the isogeny defined
by $\hat{M}$ which is \'etale since we work over $\C$. 
 By [\ref{Mukai}, Proposition 3.11],
$$
\phi^*( {M}^\vee) \cong \oplus^{h^0(\hat{M})} \hat{M}
$$

Let $\hat{f}: \hat{X}:=X \times_A \hat{A} \to \hat{A}$ be the map obtained by
base change, and $\varphi: \hat{X} \to X$ the induced map.  Let
$\mathcal{G}_s:=\hat{f}_*\mathcal{O}_{\hat{X}}(\varphi^*sID)$. 
Now 
$$
H^i(A,\mathcal{F}_s\otimes {M}^\vee)=0 ~~\mbox{for all $i>0$}
$$
if and only if 
$$
H^i(X,\mathcal{O}_X(sID)\otimes f^*{M}^\vee)=0 ~~\mbox{ for all $i>0$} 
$$
and this in turn holds if 
$$
H^i(\hat{X},\mathcal{O}_{\hat{X}}(\varphi^* sID)\otimes \varphi^*f^*{M}^\vee)=0~~\mbox{for all $i>0$} 
$$
because $\mathcal{O}_X$ is a direct summand of $\varphi_*\mathcal{O}_{\hat{X}}$ and the 
sheaves involved are locally free. The latter vanishing is equivalent to 
$$
H^i(\hat{A},\mathcal{G}_s\otimes \phi^*{M}^\vee)
=\oplus^{h^0(\hat{M})} H^i(\hat{A},\mathcal{G}_s\otimes \hat{M})=0~~\mbox{for all $i>0$} 
$$
which holds by Step 1.
\end{proof}

\begin{rem}\label{rem-v0}
We will abuse notation and use the same notation for a Cartier divisor and its 
associated line bundle.
Going a little bit further in the above proof, one can have a
stronger statement, that is, if $V^0(\mathcal{F}_s)$ has positive dimension, 
then we can choose $P$ so that $\kappa(D+f^*P)>0$. Indeed fix an $s \ge 2$ and let
$W_s$ be an irreducible component of $V^0(\mathcal{F}_s)$. Clearly
there is a natural multiplication map  $\pi_s\colon W_s \times W_s \to
V^0(\mathcal{F}_{2s})$ which sends $(Q,Q')$ to $Q\otimes Q'$. 
Let $W_{2s}$ be the irreducible component
containing the image. Then, $\dim W_s\le \dim W_{2s}$ since for each 
$Q\in W_s$, we have $\pi_s(Q\times W_s)=Q\otimes W_s\subseteq W_{2s}$.
Thus, $\dim W_{ks}$ is independent of $k$ for $k\gg 0$. For such $k$ 
we have: if $Q,Q'\in W_{ks}$, then $Q\otimes W_{ks}=Q'\otimes W_{ks}$ 
which implies that 
$$
G=W_{ks}\otimes W_{ks}^{-1}:=\{Q\otimes Q'^{-1} \mid Q,Q'\in W_{ks}\}
$$
is an abelian subgroup of $\hat{A}$, $W_{ks}\otimes G=W_{ks}$ and that 
$W_{ks}=Q\otimes G$ for any $Q\in W_{ks}$. Therefore, $W_{ks}$ is a
translation of an abelian subvariety of positive dimension. It
follows that for certain $Q\in W_{ks}$ and $Q'\in G$ the two arrows 
$$
\xymatrix{
 & H^0(\mathcal{F}_{ks}\otimes Q)\otimes H^0(\mathcal{F}_{ks}\otimes Q) \ar[d] \\
H^0(\mathcal{F}_{ks}\otimes Q\otimes Q')\otimes H^0(\mathcal{F}_{ks}\otimes Q\otimes Q'^{-1}) \ar[r] 
  &  H^0(\mathcal{F}_{2ks}\otimes Q^2) 
}
$$
produce two sections of $\mathcal{F}_{2ks}\otimes Q^2$ hence
$\kappa(D+f^*P)>0$ for $P=\frac{1}{ksI}Q$.
\end{rem}

%%%%%%%%%%%%%%%%%%%%%%
%%%%%%%%%%%%%%%%%%%%%%
\vspace{0.5cm}
\section{Proof of the main theorem}

\begin{rem}\label{rem-Ueno-fib}
Let $Z$ be a normal projective variety and let $\pi\colon Z\to A$ be 
a generically finite morphism to an abelian variety $A$. Let $Z\to Z'\to A$ be the 
Stein factorization of $\pi$. 
Then by Kawamata [\ref{Kawamata-abelian}, Theorem 13] (also see Ueno [\ref{Ueno}, Theorem 3.10]), 
there exist an abelian subvariety $A_1\subseteq A$,  \'etale covers $\tilde{Z'}\to Z'$ and $\tilde{A_1}\to A_1$, 
and a normal projective variety ${Z_2}'$ such that 

$\bullet$  we have a finite morphism ${Z_2}'\to A_2:=A/A_1$,

$\bullet$ $\tilde{Z'}\simeq \tilde{A_1}\times Z_2'$,

$\bullet$ $\kappa(Z)=\kappa(Z')=\kappa(Z_2')=\dim Z_2'$.\\
\end{rem}

\begin{proof}(of Theorem \ref{t-mmodel})
\emph{Step 1.}
We use induction on dimension of $X$. By assumptions, there is a generically finite morphism 
$Z\to A$ to an abelian variety.  Since $K_X+B$ is big$/Z$, it is also big$/A$.
So, by Theorem \ref{t-nv}, $K_X+B\equiv D$ for some $\Q$-divisor $D\ge 0$. By 
Campana-Koziarz-P\v{a}un [\ref{CKP}] or Kawamata [\ref{Kawamata-v=0}],  
we may assume that $K_X+B\sim_\Q D$. After taking appropriate resolutions, we may 
in addition assume that $Z$ is smooth, 
 $X$ is smooth, and that $\Supp (B+D)$ has simple 
normal crossing singularities.\\ 

\emph{Step 2.}
Let $N\ge 0$ be the $\Q$-divisor such that $\Supp N=\Supp D$, $(X,B+N)$ is lc, and 
$\Supp \rddown{B+N}=\Supp D$. Let $\Delta:=B+N$. 
By Section \ref{rem-LMMP-abelian}, we can run an LMMP on $K_X+\Delta$ which ends up with a 
log minimal model $(Y,\Delta_Y)$. Moreover, $K_Y+\Delta_Y$ is semi-ample$/A$. 
Let $Y\to Y'/A$ be the contraction defined by $K_Y+\Delta_Y$. 

Now assume that every component of $\rddown{\Delta_Y}$ is contracted$/Y'$ 
which is equivalent to saying that every component of $D_Y$ is contracted$/Y'$ 
because $\Supp \rddown{\Delta_Y}=\Supp D_Y$.   
In this case
$$
K_{Y'}+\Delta_{Y'}\sim_\Q D_{Y'}+N_{Y'}=0
$$
which implies that $K_Y+\Delta_Y\sim_\Q 0$ because $K_Y+\Delta_Y$ is the pullback of 
$K_{Y'}+\Delta_{Y'}$. Therefore, 
$$
\kappa(D_Y+N_Y)=\kappa_\sigma(D_Y+N_Y)=0
$$ 
where $\kappa_\sigma$ 
denotes the numerical Kodaira dimension defined by Nakayama. This in turn implies that 
$$
\kappa(D+N)=\kappa_\sigma(D+N)=0
$$ 
which gives $\kappa(D)=\kappa_\sigma(D)=0$ because $\Supp N=\Supp D$. 
Therefore, $(X,B)$ has a good log minimal model (cf. [\ref{Gongyo}]). 
Moreover, since $D$ is big$/A$ and since 
$D$ is contracted$/Y'$, $X\to Z$ is generically finite, that is,  $\dim F=0$. 
Thus, 
$$
\kappa(K_X+B)\ge \kappa(X)\ge \kappa(Z)=\kappa(K_F+B_F)+\kappa(Z)
$$ 
So, from now on we may assume that 
there is a component 
$S$ of $\rddown{\Delta}$ such that $S_Y$ is not contracted$/Y'$.\\ 

\emph{Step 3.}
Choose a small rational number $\epsilon>0$ and put 
$$
K_{X}+\Gamma:=K_{X}+\Delta-\epsilon(\rddown{\Delta}-S)
$$
Then $(X,\Gamma)$ and $(Y,\Gamma_Y)$ are both plt with $\rddown{\Gamma}=S$ 
and $\rddown{\Gamma_Y}=S_Y$.
Let $Y''$ be a log minimal model of $K_{Y}+\Gamma_Y$ over $Y'$. Such a model exists 
because 
$K_{X}+\Gamma\sim_\Q M$ for some $M\ge 0$ with $S\subset \Supp M\subset \Supp \Gamma$ 
which allows us to reduce the problem to the klt case similar to Section \ref{rem-LMMP-abelian}. 
In particular, this also shows that   
$K_{Y''}+\Gamma_{Y''}$ is semi-ample$/Y'$ by the base point free theorem. 
Therefore, since $K_{Y'}+\Delta_{Y'}$ is ample$/A$,
by choosing $\epsilon$ to be small enough, we can assume that $K_{Y''}+\Gamma_{Y''}$ is 
semi-ample$/A$.\\ 

\emph{Step 4.}
By the arguments of 
Section \ref{rem-LMMP-abelian}, $K_{Y''}+\Gamma_{Y''}$ is nef globally. This implies that 
$N_\sigma(K_{Y''}+\Gamma_{Y''})=0$ which in particular means that $S_{Y''}$ is not 
a component of $N_\sigma(K_{Y''}+\Gamma_{Y''})$. Moreover, by construction, 
$(K_{Y'}+\Delta_{Y'})|_{S_{Y'}}$ is ample$/A$ hence $(K_{Y''}+\Delta_{Y''})|_{S_{Y''}}$ is big$/A$ 
which implies that  
$(K_{Y''}+\Gamma_{Y''})|_{S_{Y''}}$ is big$/A$ as $\epsilon$ is sufficiently small. 
Also, since $K_{Y''}+\Gamma_{Y''}$ is a plt log divisor, $(K_{Y''}+\Gamma_{Y''})|_{S_{Y''}}$ is a klt log divisor.
Thus by induction 
$\kappa(K_{Y''}+\Gamma_{Y''})|_{S_{Y''}})\ge 0$. On the other hand, 
$K_{Y''}+\Gamma_{Y''}\sim_\Q M_{Y''}\ge 0$ with $S_{Y''}\subset \Supp M_{Y''}\subset \Supp \Gamma_{Y''}$. 
Therefore, by the extension result [\ref{DHP}, Theorem 1.8], 
$$
H^0(Y'',m(K_{Y''}+\Gamma_{Y''}))\to H^0(S_{Y''},m(K_{Y''}+\Gamma_{Y''})|_{S_{Y''}})
$$
is surjective for any sufficiently divisible $m>0$. This means that 
$\kappa(K_{Y''}+\Gamma_{Y''})\ge 1$ hence 
$\kappa(K_{Y''}+\Delta_{Y''})\ge 1$ so $\kappa(K_{Y'}+\Delta_{Y'})\ge 1$ which in turn implies that
 $\kappa(K_{X}+\Delta)\ge 1$. Thus  $\kappa(K_{X}+B)\ge 1$ since $\Delta=B+N$ and 
 $N$ is supported in $D$.\\

\emph{Step 5.}
Since  $\kappa(K_{X}+B)\ge 1$, $K_X+B$ has a non-trivial Iitaka fibration $g\colon X\bir V$ 
which we may assume to be a morphism.
Let $G$ be a general fibre of $g$. Then by the definition of Iitaka fibration, 
$\kappa(K_G+B_G)=0$ where $K_G+B_G=(K_X+B)|_G$. 
Since $K_X+B$ is big$/A$, we can write 
$$
K_X+B\sim_\Q H+E/A
$$ 
where $H$ is ample and $E\ge 0$. Thus 
$$
K_G+B_G\sim_\Q H|_G+E|_G/A
$$ 
which shows that $K_G+B_G$ is also big$/A$. So $K_G+B_G$ is big$/P$ where $P$ is the 
normalization of the image of $G$ in $A$.
Now since $(G,B_G)$ is klt, $K_G+B_G$ is big$/P$, and $P$ has maximal albanese 
dimension, by induction, $(G,B_G)$ has a 
good log minimal model. Moreover, again by induction, 
$$
0=\kappa(K_G+B_G)\ge \dim G-\dim P+\kappa(P)
$$
which implies that $\dim G-\dim P=0$ hence $G\to A$ is  
 generically finite. Therefore, the general fibres $G$ of $g$ intersect the general 
 fibres $F$ of $f$ at at most  
finitely many points. Thus $F\to V$ is also generically finite which implies that 
$$
\kappa(K_X+B)=\dim V\ge \dim F
$$\

\emph{Step 6.}
Let $Z', \tilde{Z'}, Z_2', \tilde{A_1}$, etc, be as in Remark \ref{rem-Ueno-fib} 
which are constructed for the morphism $Z\to A$. 
Let $\tilde{Z}={Z}\times_{Z'}\tilde{Z'}$ 
and  $\tilde{X}={X}\times_{Z'}\tilde{Z'}$. Put $Z_1':=\tilde{A_1}$. 
The induced morphisms  $\tilde{Z'}\to Z'$, $\tilde{Z}\to Z$, and $\tilde{X}\to X$ are all \'etale.
By replacing $Z'$ with  $\tilde{Z'}$, $Z$ with  $\tilde{Z}$, and  $X$ with  $\tilde{X}$, we can 
assume that $Z'=Z_2'\times Z_1'$ where $Z_1'$ admits an \'etale finite morphism onto an 
abelian subvariety $A_1\subseteq A$ and $Z_2'$ admits a finite morphism to $A_2=A/A_1$. 

On the other hand, 
as pointed out above, $(G,B_G)$ has a good log minimal model, say $({G}',B_{{G}'})$. 
Since $Z_2$ admits a finite morphism into an abelian variety, it does not contain any rational 
curve hence in particular the induced map ${G}'\bir Z_2$ is a morphism. 
Moreover, since $\kappa(K_G+B_G)=0$, 
$K_{{G'}}+B_{{G}'}\sim_\Q 0$. 
Applying the canonical bundle formula [\ref{Ambro-adjunction}], we deduce that $\kappa(Q)=0$ where 
the contraction ${G}'\to Q$ is given by the Stein factorization of ${G}'\to Z_2$. Since $Z_2$ is of general type, 
$Q$ is a point otherwise $Z_2$ would be covered by a family of subvarieties of Kodaira dimension zero which is 
not possible. Therefore, the general fibres $G$ of $g$ map into the fibres of $Z\to Z_2$. 

Since $G\to A$ is generically finite, $G\to Z$ is also generically finite hence  
$\dim G$ is at most the dimension of a general fibre of $Z\to Z_2$. Therefore,
$$
\dim Z-\kappa(Z)=\dim Z-\dim Z_2\ge \dim G 
$$
which implies that 
$$
\dim Z-\dim G\ge \kappa(Z)
$$
hence 
$$
\kappa(K_X+B)=\dim V=\dim X-\dim G
$$
$$
= \dim Z+\dim F-\dim G \ge \dim F+\kappa(Z)
$$\

\emph{Step 7.}
Finally, we will construct a good log minimal model for $(X,B)$. 
Run an LMMP$/V$ on $K_X+B$ with scaling of some ample divisor. 
Since $\kappa(K_G+B_G)=0$ and since $(G,B_G)$ has a good log minimal model, the LMMP terminates 
near $G$ hence we arrive 
at a model ${X'}$ such that $(K_{{X'}}+B_{{X'}})|_{{G'}}\sim_\Q 0$ 
 where $G'$ (by abuse of notation) denotes the birational transform of $G$ (cf. [\ref{B-lc-flips}, Theorem 1.9]). 
But then continuing the LMMP we end up with a good log minimal model of $(X,B)$ over $V$ 
(cf. [\ref{B-lc-flips}, Theorem 1.5]). Denoting the minimal model again by ${X'}$, 
the semi-ampleness of $K_{{X'}}+B_{{X'}}$ over $V$ gives a contraction ${X'}\to {T'}/V$ 
such that ${T'}\to V$ is birational and that 
$K_{{X'}}+B_{{X'}}\sim_\Q 0/{T'}$. Therefore  $(X,B)$ has a 
good log minimal model (cf. [\ref{BH}, Proposition 3.3]).\\
\end{proof}

%%%%%%%%%%%%%%%%%%%%%%%%%%%%%%
\section{Relation with the log Iitaka conjecture over abelian varieties}

Let $(X,B)$ and $X\to Z$ be as in Conjecture \ref{conj-iitaka-pairs} and assume that 
$Z$ has maximal albanese dimension. So, there is a generically finite map $Z\to A$ 
into an abelian variety.

Let us explain how one might use the method of the proof of Theorem \ref{t-mmodel} to 
show that 
$$
\kappa(K_X+B)\ge \kappa(K_F+B_F)+\kappa(Z)
$$
If one can show that $\kappa(K_X+B)\ge 1$, then the rest of the proof would consist of 
some relatively easy inductive arguments. 

 Assume that $\kappa(K_X+B)\le 0$.
First, we borrow the main idea
of [\ref{Birkar}] to consider the relative Iitaka fibration of $K_X+B$ over $Z$ 
which was used 
to prove the Iitaka conjecture in dimension six. Next, we apply 
Fujino-Mori [\ref{Fujino-Mori}] to get a klt pair $(X',B')$, a morphism
$X'\to Z$, and a nef $\Q$-divisor $L'$
such that $K_{X'}+B'+L'$ is big$/Z$, 
$$
\kappa(K_X+B)=\kappa(K_{X'}+B'+L')
$$
and $\kappa(K_F+B_F)$ is equal to the dimension of a general fibre of $X'\to A$.
If $\kappa(K_F+B_F)=0$, then $X'\to A$ is generically finite and the proof 
in this case is relatively easy. So, we can assume $\kappa(K_F+B_F)\ge 1$.

By Theorem \ref{t-nv} and results of 
Campana-Koziarz-P\v{a}un [\ref{CKP}, Theorem 1] and Kawamata [\ref{Kawamata-v=0}]
we get a nonvanishing 
$$
K_{X'}+B'+L'\sim_\Q D'\ge 0
$$
Just as in the proof of Theorem \ref{t-mmodel}, we can construct a plt pair $(X',\Gamma')$ so that
$$
\kappa(K_{X'}+B'+L')=\kappa(K_{X'}+\Gamma'+L')
$$
$S':=\rddown{\Gamma'}\neq 0$, $S'$ is irreducible, $S'\subseteq \Supp D'\subseteq \Supp \Gamma'$, and 
$(K_{X'}+\Gamma'+L')|_{S'}$ is big$/Z$.

Now, we apply the nonvanishing again to get nonzero sections in
$$
H^0(S',m(K_{X'}+\Gamma'+L')|_{S'})
$$
for some sufficiently divisible $m>0$. If we could extend a section, then we would be done. 
In general, one cannot extend sections if $L'$ is an arbitrary nef divisor. 
But in our situation $L'$ is expected to have strong positivity properties such as 
being semi-positive in the analytic sense. If one can show that such a semi-positivity 
property holds, then it seems that some section can be extended from $S'$. At the moment 
we are not able to show that $L'$ is semi-positive.

%%%%%%%%%%%%%%%%%%%%%%%%%%%%%%%

%%%%%%%%%%%%%%%%%%%%%%

\vspace{2cm}

\flushleft{DPMMS}, Centre for Mathematical Sciences,\\
Cambridge University,\\
Wilberforce Road,\\
Cambridge, CB3 0WB,\\
UK\\
email: c.birkar@dpmms.cam.ac.uk
\vspace{0.3cm}

National Center for Theoretical Sciences, Taipei Office, and\\
Department of Mathematics,\\
National Taiwan University,\\
Taipei 106, Taiwan\\
email: jkchen@math.ntu.edu.tw

\vspace{0.5cm}


\begin{thebibliography}{99}

\bibitem{}\label{Ambro-adjunction}  {F. Ambro; {\emph{The moduli b-divisor of an lc-trivial fibration.}}
Compositio Math. 141 (2005), no. 2, 385-403.}

\bibitem{}\label{B-lc-flips} C. Birkar; \textit{Existence of log canonical flips and a special LMMP.}
Publ. Math. de l'IH\'ES, Volume 115, Issue 1 (2012), 325-368.

\bibitem{}\label{Birkar} {C. Birkar; {\emph{The Iitaka conjecture Cn,m in dimension six}}.
Compositio Math. (2009), 145, 1442-1446. }


\bibitem{}\label{BCHM}  {C. Birkar, P. Cascini, C. D. Hacon, J. M$^c$Kernan; {\emph{Existence of minimal models
for varieties of log general type.}}  J. Amer. Math. Soc. 23 (2010), 405-468. }

\bibitem{}\label{BH} {C. Birkar, Z. Hu; {\emph{Log canonical pairs with good augmented base loci}}.
To appear in Compositio Math. arXiv:1305.3569v2}

\bibitem{}\label{CCP} {F. Campana, J.A. Chen, T. Peternell; 
{\emph{On strictly nef divisors.}} Math. Ann., 342, (2008), 565--585.}

\bibitem{}\label{CKP} {F. Campana, V. Koziarz, M. P\v{a}un; {\emph{Numerical character of the effectivity of adjoint line bundles}}.    Ann. Inst. Fourier 62, (2012), 107--119.}

\bibitem{}\label{Chen-Hacon} {J.A. Chen, C. Hacon; {\emph{Kodaira dimension of irregular varieties}}.
Invent. Math. Volume 186, Issue 3 (2011), 481--500. }

\bibitem{}\label{DHP} J.P. Demailly, C. D. Hacon, M. P\v{a}un; 
\emph{Extension theorems, Non-vanishing and the existence of good minimal models.} 	
Acta Math. Volume 210, Issue 2 (2013), 203--259.


\bibitem{}\label{Fujino-Mori}  {O. Fujino, S. Mori; {\emph{A canonical bundle formula.}} 
J. Differential Geometry 56 (2000), 167-188.}

\bibitem{}\label{Gongyo} {Y. Gongyo; {\emph{On the minimal model theory for dlt pairs of numerical log Kodaira dimension zero}}.
Math. Res. Lett. 18 (2011), no. 5, 991--1000.}

\bibitem{}\label{Ha04} C. D. Hacon; \emph{A derived category approach to generic vanishing.}
J. Reine Angew. Math. {575}, 173-187 (2004).


\bibitem{}\label{Iitaka-book} S. Iitaka; \emph{Algebraic geometry: an introduction to the
birational geometry of algebraic varieties.}  Springer-Verlag, 1982.

\bibitem{}\label{Kawamata-v=0}  {Y. Kawamata; {\emph{On the abundance theorem in the case $\nu=0$.}} arXiv:1002.2682v3.}

\bibitem{}\label{Kawamata-abelian}  {Y. Kawamata; {\emph{Characterization of abelian varieties.}} 
Compositio Math., volume 43, no 2 (1981), 253-276.}  
 
\bibitem{}\label{Kollar-Mori}  {J. Koll\'ar, S. Mori; {\emph{Birational Geometry of Algebraic Varieties.}} Cambridge University
Press (1998).} 

\bibitem{}\label{Mukai} S. Mukai; \emph{Duality between $D(X)$ and $D(\hat X)$ with application to Picard sheaves.}
Nagoya Math. J. {\bf 81}, 153-175 (1981)

\bibitem{}\label{Ueno} K. Ueno; \emph{Classification of algebraic varieties I.} 
Compositio Math. Vol. 27, Fasc. 3, (1973), 277-342.
\end{thebibliography}
\end{document}